\documentclass[12pt,reqno]{amsart}
\usepackage{amssymb,epsf}
\usepackage{amstext,amsmath,amsfonts,amscd,amsthm,graphicx}
\usepackage{upref}
\usepackage{epsfig}
\usepackage[utf8]{inputenc}
\usepackage{latexsym}
\usepackage{eucal}
\usepackage{wrapfig}
\usepackage{color}
\usepackage{enumerate}
\theoremstyle{definition}

\newtheorem{rmk}{Remark}[section]
\newtheorem{con}{Construction}[section]
\theoremstyle{plain}
\newtheorem{thm}{Theorem}[section]
\newtheorem{lm}{Lemma}[section]

\begin{document}
\author{Barbora Voln\'{a}}
\address{Mathematical Institute, Silesian University in Opava, \newline \indent Na Rybn\'{i}\v{c}ku 1, 746 01 Opava, Czech Republic}
\email{Barbora.Volna@math.slu.cz}
\keywords{Devaney chaos, continuous group action, uniform space, topological transitivity, sensitivity}
\subjclass[2010]{37B05, 54H11, 37D45, 34A60}

\title[Devaney chaos for continuous group action on Hausdorff uniform space]{On definition of Devaney chaos for a continuous group action on a Hausdorff uniform space}

\begin{abstract}
We show that the existence of a dense set of periodic points for a  topologically transitive non-minimal continuous group action on a Hausdorff uniform space with an infinite acting group does not necessarily imply a sensitive dependence to the initial conditions in such a system. This leads to define the chaos in the sense of Devaney for a continuous group action on a Hausdorff uniform spaces with an infinite acting group in the original way, i.e. a non-minimal topologically transitive and sensitive system with a dense set of periodic points is a chaotic system in the sense of Devaney.
\end{abstract}

\maketitle

\section{Introduction}

The chaos is one of the most studied terms in the field of the dynamical systems and of other related branches in mathematics. There exist many definitions of chaos in a dynamical system. The sensitivity on the initial conditions is the core of many of these definitions. We can find the mention about the sensitivity already in the familiar description of the Lorenz attractor from 1963, \cite{lorenz}. R.L. Devaney \cite{devaney} in 1989 defined the chaos in the following way. Let $V$ is a set. A map $f: V \rightarrow V$ is chaotic provided $f$ has a sensitive dependence on the initial conditions, $f$ is topologically transitive and $f$ has a dense set of periodic points. Today, the chaos defined in this way is said to be the Devaney chaos or the chaos in the sense of Devaney. Generally, we can see that such defined chaos is considered for a discrete-time dynamical system, or for a cascade, or one can say for a continuous group action with a finite acting group. In 1992 and 1993 in \cite{banks_brooks_cairns_davis_stacey}, \cite{glasner_weiss} and \cite{silverman} authors showed that for such systems the sensitivity follows from another two properties (i.e. from the topological transitivity and from the existence of a dense set of periodic points). And in \cite{glasner_weiss} authors pointed out the importance of the property of the non-minimality in such systems. So, we can say that a non-minimal topologically transitive discrete dynamical system with a dense set of periodic points is a chaotic system in the sense of Devaney.

Naturally, many authors react on such observations and try to extend these results for a continuous (semi-) group action with an infinite acting group or for a (semi-) flow on a metric space or on a Polish space  \cite{dai}, \cite{kontorovich_megrelishvili}, \cite{polo}, \cite{wang_long_fu}, \cite{wang_yin_yan}, or more general on a Hausdorff uniform space \cite{ceccherini-silberstein_coornaert}, \cite{dai_tang}, \cite{schneider_kerkhoff_behrisch_siegmund}. We focus on the definition of the Devaney chaos for a continuous group action on a Hausdorff uniform space with an infinite acting group. We provide an example of a continuous group action on a Hausdorff uniform space with an infinite acting group which is topologically transitive, has a dense set of periodic points, is not minimal and is not sensitive on the initial conditions.  Although in \cite{ceccherini-silberstein_coornaert}, \cite{dai_tang}, \cite{schneider_kerkhoff_behrisch_siegmund} authors declared the opposite, generally, the sensitivity does not follow from the topological transitivity and from the existence of a dense set of periodic points for a non-minimal continuous group action on a Hausdorff uniform space with an infinite acting group. We see that the definition of the Devaney chaos for such a continuous dynamical system should be remained the same as the original one.

\section{Main result}
 
We show that a non-minimal topologically transitive continuous group action on a Hausdorff uniform space with an infinite acting group and with a dense set of periodic points need not be sensitive on the initial conditions. For a discrete dynamical system and for a continuous dynamical system on a metric or metrizable space the topological transitivity, the existence of a dense set of periodic points and the minimality ensure the sensitivity on the initial conditions, see \cite{banks_brooks_cairns_davis_stacey}, \cite{dai}, \cite{glasner_weiss}, \cite{kontorovich_megrelishvili}, \cite{polo}, \cite{silverman}, \cite{wang_long_fu}, \cite{wang_yin_yan}. Many authors tried to prove the same property for a continuous group action on a Hausdorff uniform space with an infinite acting group, see e.g. \cite{ceccherini-silberstein_coornaert}, \cite{dai_tang}, \cite{schneider_kerkhoff_behrisch_siegmund}, but we provide an example of such a system where the sensitivity is not ensured by the topological transitivity, by the non-minimality and by the existence of a dense set of periodic points.

\begin{thm}
\label{thm_main_result}
There exists a continuous group action on a Hausdorff uniform space with an infinite acting group which
\begin{itemize}
\item is topologically transitive,
\item is not minimal,
\item has a dense set of periodic points,
\item is not sensitive to the initial conditions,
\item and is not equicontinuous.
\end{itemize}
\end{thm}

Proof of this theorem is located in the section \ref{section_proofs}.

\section{Background, definitions and notations}

Firstly, we focus on the well-known definition and background of the uniform space, see e.g. \cite{dugundji}. Let $X$ be a set, let $U,V \subset X \times X$, let $\Delta_X :=\{ (x,x) \mid x \in X \}$ be the diagonal in $X \times X$, let $U^{-1}:=\{ (y,x) \mid (x,y) \in U \}$, let $U \circ V := \{ (x,z) \mid \exists y \textmd{ such that } (x,y) \in U \textmd{ and } (y,z) \in V  \}$. A \textit{uniform structure} in a set $X$ is a family $\mathcal{V}$ of subsets of $X \times X$ such that
\begin{enumerate}[(i)]
\item if $U \in \mathcal{V}$ then $\Delta_X \in \mathcal{V}$;
\item if $U_1,U_2 \in \mathcal{V}$ then there is a $V \in \mathcal{V}$ such that $V \subset U_1 \cap U_2$;
\item if $U \in \mathcal{V}$ then there is a $V \in \mathcal{V}$ such that $V \circ V^{-1} \subset U$;
\item if $U \in \mathcal{V}$ and $V \subset U$ then $V \in \mathcal{V}$.
\end{enumerate}
The element of $\mathcal{V}$ is called an \textit{entourage} of the uniform structure. A family of subsets of $X \times X$ satisfying the conditions (i), (ii), (iii) is called a \textit{uniformity}. For each $x \in X$ let $U[x]:=\{ y \mid (x,y) \in U \}$. One can say that two points are $U$-closed whenever $(x,y) \in U$, or equivalently, $y \in U[x]$. From any uniformity $\mathcal{V}$ in $X$ a topology $\mathcal{T}(\mathcal{V})$ in $X$ is derived by taking the family $\{ U[x] \mid U \in \mathcal{V}, x \in X \}$ as the basis of the topology $\mathcal{T}(\mathcal{V})$. A space $X$ with a topology $\mathcal{T}(\mathcal{V})$ derived from a uniformity is called a \textit{uniform space}. For a detailed exposition of uniform spaces see e.g. \cite{dugundji}.

Secondly, we concentrate on definitions and notations related to a \textit{continuous dynamical system} given by a continuous group action with an infinite acting group. Let $X$ be a set equipped with a group action
$$T: G \times X \rightarrow X$$
where $G$ is a group and 
\begin{enumerate}
\item $T(r,T(s,x))=T(rs,x)$ for $r,s \in G$ and $x \in X$,
\item $T(e,x)=x$ for $x \in X$ where $e$ is an identity element in $G$,
\end{enumerate}
see e.g. \cite{dai_tang}, \cite{schneider_kerkhoff_behrisch_siegmund}. We consider that $X$ is a topological space, a group action $T$ on $X$ is continuous and an acting group is infinite. According to \cite{ceccherini-silberstein_coornaert}, an action $T$ of a group $G$ on a topological space $X$ is continuous if for each $g \in G$ the map $x \mapsto T(g,x)$ is continuous on $X$. According to some authors, this system is called a \textit{topological flow}, see e.g. \cite{dai_tang}. The set $Tx:=\{ T(g,x) \mid g \in G \}$ is the \textit{orbit} of the point $x \in X$ and the set $S_T(x):=\{ g \in G \mid T(g,x)=x \}$ is the \textit{stabilizer} of $T$ at the point $x \in X$, see e.g. \cite{ceccherini-silberstein_coornaert}, \cite{schneider_kerkhoff_behrisch_siegmund}.

Finally, we recollect some definitions and background related to the properties of a group action. We say that a group action $T$ on $X$ is \textit{topologically transitive} if for every pair $(Y_1,Y_2)$ of non-empty open sets $Y_1,Y_2 \subseteq X$ there exist $g \in G$ such that $T(g,Y_1) \cap Y_2 \neq \emptyset$, see \cite{ceccherini-silberstein_coornaert}, \cite{dai_tang}, \cite{schneider_kerkhoff_behrisch_siegmund}. Let $T(g^{-1},Y_2):=\{ x \in X \mid T(g,x) \in Y_2 \}$. Since $T(g,Y_1 \cap T(g^{-1},Y_2))=T(g,Y_1) \cap Y_2$ it is equivalent to say that $Y_1 \cap T(g^{-1},Y_2) \neq \emptyset$ for some $g \in G$, see \cite{dai_tang}, \cite{kontorovich_megrelishvili}. We say that a group action $T$ on $X$ is \textit{minimal} if $Tx$ is dense in $X$ for every $x \in X$, i.e. $\overline{Tx}=X$ for each $x \in X$, see \cite{kontorovich_megrelishvili}, \cite{schneider_kerkhoff_behrisch_siegmund}. Now, we focus on the concept of the periodicity. According to \cite{schneider_kerkhoff_behrisch_siegmund}, a point $x \in X$ is said to be \textit{periodic} with respect to $T$ if the stabilizer $S_T(x)$ is right syndetic in G. And, a subset $S \subseteq G$ is called \textit{right syndetic} in G if there exist a compact subset $K \subseteq G$ such that $K^{-1}S:=\bigcup_{k \in K}k^{-1}S=G$ where $k^{-1}S:=\{ t \in G \mid kt \in S \}$. In \cite{dai_tang}, authors provide the definition of the periodicity which is much weaker then in \cite{schneider_kerkhoff_behrisch_siegmund}. According to \cite{dai_tang}, a point $x \in X$ is called \textit{periodic} with respect to $T$ if the stabilizer $S_T(x)$ is syndetic in $G$. And, a subset $S \subseteq G$ is called \textit{syndetic} in $G$ if there is a compact subset $K \subseteq G$ such that $Kt \cap S \neq \emptyset ~\forall t \in G$ where $Kt := \{ kt \mid k \in K \}$. At the end, we recall the definition of a sensitivity and an equicontinuity for the considered system. A group action $T$ on a uniform space $X$ \textit{has a sensitive dependence on the initial conditions}, more briefly a group action is \textit{sensitive}, provided there is an entourage $U \in \mathcal{V}$ such that for all point $x \in X$ and every neighbourhood $Y$ of $x$ in $X$ there exist a point $y \in Y$ and an element $g \in G$ such that $(T(g,x),T(g,y)) \notin U$, \cite{ceccherini-silberstein_coornaert}, \cite{dai_tang}, \cite{schneider_kerkhoff_behrisch_siegmund}. Such an entourage is called a \textit{sensitivity entourage}, see \cite{ceccherini-silberstein_coornaert}. Contrarily, a group action $T$ on a uniform space $X$ is called \textit{equicontinuous at a point} $x \in X$ if for any entourage $U \in \mathcal{V}$ there exists a neighbourhood $Y$ of $x$ in $X$ such that $T(g,Y) \subseteq U[T(g,x)]$ for all $g \in G$. A group action $T$ on a uniform space $X$ is \textit{equicontinuous} if $T$ is equicontinuous at each point $x \in X$, see e.g. \cite{schneider_kerkhoff_behrisch_siegmund}.

\section{Proof}
\label{section_proofs}

In this section, we give an example of a continuous group action on a Hausdorff uniform space with an infinite acting group which is topologically transitive, is not minimal, has a dense set of periodic points, is not sensitive to the initial conditions and is not equicontinuous.

\begin{con}
\label{con_system}
We consider the differential inclusion in $\mathbb{R}^2$ given by
\begin{equation}
\label{eq_dif_inclusion}
\dot x \in \{f_1(x),f_2(x)\}
\end{equation}
where $\{f_1,f_2\}$ is a set-valued map that associates a set $\{f_1(x),f_2(x)\}$ with every point $x \in \mathbb{R}^2$ and $f_{1,2}:\mathbb{R}^2 \rightarrow \mathbb{R}^2$ are linear functions, see also \cite{smirnov}, \cite{volna_2}. A solution of such a differential inclusion is an absolutely continuous function $x: \mathbb{R} \rightarrow \mathbb{R}^2$ such that
\begin{equation}
\dot x(t) \in \{ f_1(x(t)),f_2(x(t)) \} \textrm{ a.e.}
\end{equation}
where $t \in \mathbb{R}$, \cite{smirnov}. The set of all solutions of (\ref{eq_dif_inclusion}) contains the solutions of the one branch $\dot x = f_1(x)$ or $\dot x = f_2(x)$, or the so-called "switching" solutions constructed by jumping from the integral curves generated by $f_1$ to the integral curves generated by $f_2$, and vice versa, in some points from $\mathbb{R}^2$, see also \cite{volna_2}. Moreover, we consider that the functions $f_{1,2}$ are such that the singular points $a^*_i$ of the differentials equation given by
\begin{equation}
\dot x = f_i(x)
\end{equation}
for $i=1,2$ are stable nodes and $a^*_1 \neq a^*_2$. Let $\varphi$, $\psi$ denote the flows generated by $f_1$, $f_2$, respectively. We depict this situation in Figure \ref{fig_construction}.
\begin{figure}[ht]
  \centering
  \includegraphics[height=2cm]{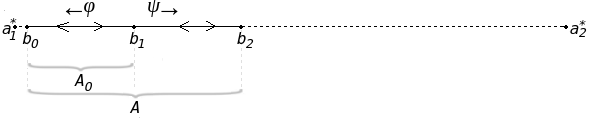}
  \caption{}
  \label{fig_construction}
\end{figure}
We can see the plane objects: singular points $a^*_1 \neq a^*_2$, the line segments with arrows representing the trajectories of the flows $\varphi$ and $\psi$. Further, we consider only the region $A$ depicted by the line segment with endpoints $b_0$ and $b_2$ located in the neighbourhood of the point $a_1$, see also \cite{volna_1}. Further, there is the line segment $A_0 \subset A$ with endpoints $b_0$ and $b_1$ with half length than $A$. Then we consider the set $D$ containing only the following switching solutions of (\ref{eq_dif_inclusion}).
\begin{itemize}
\item The set of solutions denoted by $X$ includes all solutions $x$ with the initial condition $x(0)=a$ for all $a \in A$ such that it firstly follows the integral curve generated by $f_2$ to the point $b_2$ then in this point it switches to the integral curve generated by $f_1$ and follows this integral curve until $b_0$ where it switches to the integral curve generated by $f_2$ and so on for $t \rightarrow \pm \infty$, or vice versa, i.e. it starts following the integral curve generated by $f_1$ to the point $b_0$.
\item The set of solutions denoted by $X_0$ includes all solutions $x_0$ with the initial condition $x_0(0)=a_0$ for all $a_0 \in A_0$ such that it firstly follows the integral curve generated by $f_2$ to the point $b_1$ then in this point it switches to the integral curve generated by $f_1$ and follows this integral curve until $b_0$ where it switches to the integral curve generated by $f_2$ and so on for $t \rightarrow \pm \infty$, or vice versa, i.e. it starts following the integral curve generated by $f_1$ to the point $b_0$.
\end{itemize}
We can see that $D = X \cup X_0$ and all considered solutions in $D$ are periodic. Let $\tau$ denote the period of $x \in X$. Then, the period of $x_0 \in X_0$ is $\frac{\tau}{2}$. Now, we define the topology on the set $D$ in the following way. Let $x_1, x_2 \in D$, let $t \in \mathbb{R}$ and let $d$ be the Euclidean metric in $\mathbb{R}^2$. We define the family of pseudometrics (or gauges)
\begin{equation}
\label{eq_family_of_pseudometrics}
\mathcal{D} := \{ d_t \mid t \in \mathbb{R} \}
\end{equation}
so that
\begin{equation}
d_t(x_1,x_2):=d(x_1(t),x_2(t)).
\end{equation}
The family of pseudometrics (\ref{eq_family_of_pseudometrics}) on the set $D$ is separating, i.e. for each pair of points $x_1 \neq x_2$ there is a $d_{\alpha} \in \mathcal{D}$ such that $ d_{\alpha}(x_1,x_2) \neq 0$. Let $B(y,d_t,\epsilon)$ denote the open ball of the radius $\epsilon$ centred at the point $y$ corresponding to the pseudometric $d_t$. So, we consider the topology $\mathcal{T}(\mathcal{D})$ in $D$ induced by the family of the pseudometrics $\mathcal{D}$, i.e. it is the topology having for a subbasis the family of the balls $\{ B(y,d_t,\epsilon) \mid y \in D, d_t \in \mathcal{D}, \epsilon > 0 \}$. Since the family of the pseudometrics $\mathcal{D}$ is separating our considered topology $\mathcal{T}(\mathcal{D})$ on the set $D$ is Hausdorff. We can say that the separating family $\mathcal{D}$ of gauges $d_t$ for $t \in \mathbb{R}$ is the \textit{gauge structure} for $D$ with the considered topology $\mathcal{T}(\mathcal{D})$ and this topological space is the \textit{gauge space}. For a detailed exposition of the gauge spaces see \cite{dugundji}. The family of the sets
\begin{equation}
\label{construction_uniformity}
\{(x_1,x_2) \subset D \times D \mid d_t(x_1,x_2) < \epsilon \}
\end{equation}
for all $d_t \in \mathcal{D}$ and $\epsilon > 0$ is the uniformity in $D$ determined by $\mathcal{D}$, see also \cite{dugundji}. So, $D$ is the uniform space. Now, we consider the natural $\mathbb{R}$-action on $D$
\begin{equation}
T: \mathbb{R} \times D \rightarrow D
\end{equation} 
where $T(t,y)=z$, $z(s)=y(t+s)$ for all $s \in \mathbb{R}$, see also \cite{volna_2}. We see that considered $(\mathbb{R},+)$ is the infinite group and $T$ is the group action on the Hausdorff uniform space $D$. Furthermore, the group action $T$ on $D$ is continuous. Obviously, for each $t \in \mathbb{R}$ the map $y \mapsto T(t,y)$ is continuous corresponding to the topology $\mathcal{T}(\mathcal{D})$.  
\end{con}

\begin{lm}
\label{lm_topologically_transitive}
Let $D$ be the Hausdorff uniform space specified above and $T$ be the continuous $\mathbb{R}$-action on $D$ specified above. Then $T$ is topologically transitive.
\end{lm}
\begin{proof}
Let $Y_1,Y_2 \subseteq D$ be non-empty open. Then, there are four possible situations specified below.
\begin{enumerate}[(a)]
\item $Y_1,Y_2 \subseteq X$, or $Y_1, Y_2 \subseteq X_0$. Then $Y_1,Y_2$ lies in the same orbit of each $x \in Y_1 \cup Y_2$. So, there exists $s \in \mathbb{R}$ such that $T(s,Y_1) \cap Y_2 \neq \emptyset$.
\item $Y_1 \subseteq X$ and $Y_2 \subseteq X_0$. Let $y_1 \in Y_1$ and $y_2 \in Y_2$, arbitrary. Pick $d_t \in \mathcal{D}$ and $\epsilon>0$ such that $B(y_2,d_t,\epsilon) \subseteq Y_2$. Then, there exist $s \in \mathbb{R}$ such that $d(y_1(t+s), y_2(t))=0$. Let $z:=T(s,y_1)$, then $y_1(t+s)=z(t)$ and $0=d(z(t),y_2(t))=d_t(z,y_2)=d_t(T(s,y_1),y_2)<\epsilon$. This implies that $T(s,y_1) \in Y_2$. And so, $T(s,Y_1) \cap Y_2 \neq \emptyset$. 
\item $Y_1 \subseteq X_0$ and $Y_2 \subseteq X$.  Let $y_1 \in Y_1$ and $y_2 \in Y_2$, arbitrary. Pick $d_t \in \mathcal{D}$ and $\epsilon>0$ such that $B(y_1,d_t,\epsilon) \subseteq Y_1$. Then, there exist $k \in \mathbb{R}$ such that $d(y_1(t), y_2(t+k))=0$. Let $z:=T(k,y_2)$, then $y_2(t+k)=z(t)$. Let $s:=-k$. Then $T(s^{-1},Y_2)=\{x \in D \mid T(s,x) \in Y_2 \}$. So, $z \in T(s^{-1},Y_2)$ since $T(s,z)=T(-k,z)=T(-k,T(k,y_2))=y_2$ and $y_2 \in Y_2$. Also, $z \in Y_1$ since $0=d(y_1(t),z(t))=d_t(y_1,z)<\epsilon$. This implies that $Y_1 \cap T(s^{-1},Y_2) \neq \emptyset$.
\item $Y_1 \cap X \neq \emptyset$ and $Y_1 \cap X_0 \neq \emptyset$ and $Y_2 \cap X \neq \emptyset$ and $Y_2 \cap X_0 \neq \emptyset$. Then if we pick $y_1 \in Y_1 \setminus X_0$ and $y_2 \in Y_2 \setminus X_0$ or $y_1 \in Y_1 \setminus X$ and $y_2 \in Y_2 \setminus X$ it follows to the situation (a). If we pick $y_1 \in Y_1 \setminus X_0$ and $y_2 \in Y_2 \setminus X$ it follows to the situation (b). And finally, if we pick $y_1 \in Y_1 \setminus X$ and $y_2 \in Y_2 \setminus X_0$ it follows to the situation (c).
\end{enumerate}
\end{proof}

\begin{lm}
\label{lm_non-minimality}
Let $D$ be the Hausdorff uniform space specified above and $T$ be the continuous $\mathbb{R}$-action on $D$ specified above. Then $T$ is not minimal.
\end{lm}
\begin{proof}
It is obvious that the orbit of each $x \in X$ is the whole set $X$ and the orbit of each $x_0 \in X_0$ is the whole set $X_0$. Further, we see that $\overline{X}=X \cup X_0=D$. Let $y \in X$ be such that $y(0)=b$ where $b \in A \setminus A_0$, see Figure \ref{fig_construction}. Then  $B(y,d_0,\epsilon) \cap X_0 = \emptyset$ for sufficiently small $\epsilon >0$. This implies that each $y \notin \overline{X_0}$. So, $\overline{Tx_0} \neq D$ for all $x_0 \in X_0$.
\end{proof}

\begin{lm}
\label{lm_dense_set_of_periodic_points}
Let $D$ be the Hausdorff uniform space specified above and $T$ be the continuous $\mathbb{R}$-action on $D$ specified above. Then $T$ has a dense set of periodic points.
\end{lm}
\begin{proof}
It is enough to show that each point $x \in X$ is periodic since $X$ is dense in $D$. In fact, all points of $D$ are periodic and in the following proof the reader can replace $\tau$ with $\frac{\tau}{2}$ to show that each point $x_0 \in X_0$ is periodic. The stabilizer of $T$ at each point $x \in X$ is $S_T(x)=\{ t \in \mathbb{R} \mid T(t,x)=x \}= \{ k\tau \mid k \in \mathbb{Z} \}$. We distinguish two approaches to the periodicity: (a) according to authors of \cite{schneider_kerkhoff_behrisch_siegmund} and (b) according to the authors of \cite{dai_tang} (the weaker periodicity).
\begin{enumerate}[(a)]
\item We prove that $S_T(x)$ is right syndetic in $\mathbb{R}$ according to \cite{schneider_kerkhoff_behrisch_siegmund}. The required compact subset of $\mathbb{R}$ is $K=[0,\tau]$. By definition $k^{-1}S_T(x)=\{ t \in \mathbb{R} \mid k+t \in S_T(x) \}$ for $k \in K$. And $K^{-1}S_T(x)= \bigcup_{k \in K}k^{-1}S_T(x)=\mathbb{R}$.
\item We prove that $S_T(x)$ is syndetic in $\mathbb{R}$ according to \cite{dai_tang}. The required compact subset of $\mathbb{R}$ is $K=[0,\tau]$. By definition $Kt=\{ k+t \mid k \in K \}$ for $t \in \mathbb{R}$. And $Kt \cap S_T(x) \neq \emptyset$ for every $t \in \mathbb{R}$.
\end{enumerate}
\end{proof}

\begin{lm}
\label{lm_non-sensitivity}
Let $D$ be the Hausdorff uniform space specified above and $T$ be the continuous $\mathbb{R}$-action on $D$ specified above. Then $T$ has sensitive dependence on the initial conditions.
\end{lm}
\begin{proof}
We use the proof by a contradiction. The uniformity in $D$ is given by (\ref{construction_uniformity}). Let $U=\{ (x_1,x_2) \subset D \times D \mid d_{t_U}(x_1,x_2) < \epsilon_U \}$ be the sensitivity entourage where $\epsilon_U>0$ and $t_U \in \mathbb{R}$. Let $x_{b_2}$ be the point from $D$ such that $x_{b_2}(0)=b_2$. Then, we pick a sufficiently small $\epsilon_0>0$ such that $\epsilon_0 < \epsilon_U$ and $\epsilon_0<<d(b_1,b_2)$. Then, we pick the neighbourhood $B(x_{b_2},d_0,\epsilon_0)$ of the point $x_{b_2}$. Let $y \in B(x_{b_2},d_0,\epsilon_0)$ arbitrary. Let $g \in \mathbb{R}$ be such that $(T(g,x_{b_2}),T(g,y)) \notin U$. By construction, the distance $d(x_{b_2}(s),y(s))< \epsilon_0$ for all $s \in \mathbb{R}$. So, $(T(r,x_{b_2}),T(r,y)) \in \{ (x_1,x_2) \subset D \times D \mid d_s(x_1,x_2) < \epsilon_0 \}$ for all $r,s \in \mathbb{R}$. And so, we have the contradiction that $(T(g,x_{b_2}),T(g,y)) \in U$ because $\epsilon_0<\epsilon_U$, and $(T(g,x_{b_2}),T(g,y)) \notin U$ because $U$ is the sensitivity entourage.
\end{proof}

\begin{lm}
\label{lm_non-equicontinuous}
Let $D$ be the Hausdorff uniform space specified above and $T$ be the continuous $\mathbb{R}$-action on $D$ specified above. Then $T$ is not equicontinuous.
\end{lm}
\begin{proof}
Let $x_{b_0}$ be the point from $X_0 \subset D$ such that $x_{b_0}(0)=b_0$. We show that the point $x_{b_0}$ is not equicontinuous. Let $U=\{ (x_1,x_2) \subset D \times D \mid d_0(x_1,x_2) < \epsilon_0 \}$ with a sufficiently small $\epsilon_0>0$ such that $\epsilon_0<<d(b_0,b_1)$. Let $Y$ be an arbitrary neighbourhood of the point $x_{b_0}$ and pick $d_s \in \mathcal{D}$ and $\epsilon_s>0$ such that $B(x_{b_0},d_s,\epsilon_s) \subseteq Y$. We know that each $B(x_{b_0},d_s,\epsilon_s)$ contains some points from $X_0$ and also some points from $X$. Let $B_X := B(x_{b_0},d_s,\epsilon_s) \cap X$. Then, we consider $t=\frac{\tau}{2}$. We see that $U[T(\frac{\tau}{2},x_{b_0})]=U[x_{b_0}]=\{ x \in D \mid (x_{b_0},x) \in U \}$ since $T(\frac{\tau}{2},x_{b_0})=x_{b_0}$. Pick $y \in B_X$ arbitrary and let $z:=T(\frac{\tau}{2},y)$. Finally $T(\frac{\tau}{2},B_X) \subset T(\frac{\tau}{2},Y)$ and $T(\frac{\tau}{2},B_X) \nsubseteq U[x_{b_0}]$ since $d_0(x_{b_0},z)=d(x_{b_0}(0),z(0))>\epsilon_0$. Thus $T(\frac{\tau}{2},Y) \nsubseteq U[T(\frac{\tau}{2},x_{b_0})]$.
\end{proof}

The proofs of the lemma \ref{lm_topologically_transitive}, \ref{lm_non-minimality}, \ref{lm_dense_set_of_periodic_points}, \ref{lm_non-sensitivity} and \ref{lm_non-equicontinuous} together constitute the proof of the theorem \ref{thm_main_result}.

\begin{rmk}
Our system (see the construction \ref{con_system}) is not metrizable because it does not admit a countable gauge structure. It seems that the property of the metrizability in a system with an infinite acting group is the border where the sensitivity on the initial conditions is necessary or not in the definition of the Devaney chaos for a continuous group action on a Hausdorff uniform space with an infinite acting group.
\end{rmk}

\section{Conclusion}

Our system shows that a non-minimal topologically transitive continuous group action on a Hausdorff uniform space with an infinite acting group and with a dense set of periodic points does not ensure the sensitivity on the initial conditions in such a system. So, the chaos in the sense of Devaney in such systems should be defined in the original way, i.e. a non-minimal topologically transitive and sensitive system with a dense set of periodic points. Thus, the sensitivity on the initial conditions is the necessary condition in this definition.

\section*{Acknowledgements}
The research was supported by Mathematical Institute of Silesian University in Opava, Czech Republic.

\end{document}